\documentclass[12pt, a4]{amsart}
\usepackage{amsgen,amsmath,amstext,amsbsy,amsopn,amsfonts,amssymb}
\usepackage{amsmath,amssymb,amsfonts,amsthm,enumerate}
\usepackage{mathrsfs}
\usepackage{amsthm}
\usepackage[alphabetic]{amsrefs}

\newtheorem{theorem}{Theorem}
\newtheorem{lemma}[theorem]{Lemma}

\newtheorem{corollary}[theorem]{Corollary}
\newtheorem{proposition}[theorem]{Proposition}
\newtheorem{obs}[theorem]{Observation} \newtheorem{defi}[theorem]{Definition}

\newtheorem{exa}[theorem]{Example}

\newtheorem{rem}[theorem]{Remark}

\newtheorem{rems}[theorem]{Remarks}

\def\bsq{\blacksquare\medskip}

\def\H{\mathcal H}
\def\L{\mathcal L}
\def\K{\mathcal K}
\def\P{\mathcal P}

\def\B{\mathcal B}

\def\X{\mathcal X}

\def\ZZ{{\mathbf Z}}
\def\CCC{{\mathbf C}}
\def\RRR{{\mathbf R}}
\def\QQ{\mathbf Q}
\def\RR+{{\mathbf R}^*}

\def\KK{\mathbf K}

\def\AA{\mathbf A}
\def\Sol{\mathbf{Sol}_S}
\def\X{\mathbf{Nil}_S}

\def\hU{\widehat U}

\def\ind{{\rm Ind}}

\def\eps{\varepsilon}
\def\Ga{\Gamma}
\def\ga{\gamma}
\def\La{\Lambda}
\def\la{\lambda}

\def\vfi{\varphi}

\def\tous{\qquad\text{for all}\quad}
\def\tout{\qquad\text{for all}\quad}
\def\bs{\backslash}

\def\Aut{{\rm Aut}}

\def\ind{{\rm Ind}}
\def\Ker{{\rm Ker}}
\def\Ad{{\rm Ad}}
\def\Ind{{\rm Ind}}

\def\ZC{{\rm Zc}}

\def\bs{\backslash}
\def\eps{\varepsilon}

\begin{document}

\title[Groups of automorphisms of $S$-adic nilmanifolds]{On the spectral theory of groups of automorphisms of $S$-adic nilmanifolds}
\author{Bachir Bekka}
\address{Bachir Bekka \\ IRMAR \\ UMR-CNRS 6625 Universit\'e de  Rennes 1\\
Campus Beaulieu\\ F-35042  Rennes Cedex\\
 France}
\email{bachir.bekka@univ-rennes1.fr}
\author{Yves Guivarc'h}
\address{Yves Guivarc'h \\ IRMAR \\ UMR-CNRS 6625\\ Universit\'e de  Rennes 1\\
Campus Beaulieu\\ F-35042  Rennes Cedex\\
 France}
\email{yves.guivarch@univ-rennes1.fr}

 \subjclass{37A05, 22F30, 60B15, 60G50}
   
   \begin{abstract}
 Let $S=\{p_1, \dots, p_r,\infty\}$ for prime integers $p_1, \dots, p_r.$
 Let $X$ be an $S$-adic compact nilmanifold, equipped
 with the unique translation invariant probability measure $\mu.$ We characterize
 the countable  groups $\Ga$ of automorphisms of $X$ for which the
 Koopman representation $\kappa$ on $L^2(X,\mu)$ has a spectral gap.
 More specifically, we show that $\kappa$  does not have a spectral gap
 if and only if there exists a non-trivial $\Ga$-invariant 
 quotient solenoid (that is, a finite-dimensional, connected, compact abelian group) on which $\Ga$ acts as a  virtually abelian group.
  \end{abstract}
 \maketitle

 \section{Introduction}
 \label{S-Intro}
 Let $\Ga$ be a countable group acting measurably 
on a probability space $(X,\mu)$  by measure preserving transformations.
Let $\kappa=\kappa_{X}$ denote the corresponding Koopman   representation
of $\Ga$, that is, the unitary representation of $\Ga$  on $L^2(X,\mu)$ given by 
$$
\kappa(\ga) \xi(x)= \xi (\ga^{-1}x) \tout \xi\in  L^2(X,\mu), x\in X, \ga\in \Ga.
$$
We say that 
the action $\Ga\curvearrowright  (X,\mu)$ of $\Ga$ on $(X,\mu)$ has a \textbf{spectral gap}
if the restriction $\kappa_0$  of  $\kappa$ to the $\Ga$-invariant subspace
$$
L^2_0(X,\mu)=\{\xi\in L^2(X,\mu) \ :\ \int_X \xi (x) d\mu (x)=0\}
$$
does not weakly contain the trivial representation $1_\Ga;$
equivalently, if $\kappa_0$ does not have almost invariant vectors,
that is, there is \textbf{no} sequence $(\xi_n)_n$ of unit vectors
   in  $ L^2_0(X,\mu)$  such that 
   $$\lim_n\Vert \kappa_0(\ga)\xi_n-\xi_n\Vert=0\tout \ga\in \Ga.$$
The existence of a spectral gap admits the following useful quantitative version. Let $\nu$ be a probability
measure on $\Ga$  and $\kappa_0(\nu)$  the convolution operator  defined
on $L^2_0(X,\mu)$ by 
$$
\kappa_0(\nu)\xi =\sum_{\ga\in \Ga} \nu(\ga) \kappa_0(\ga) \xi \tous \xi\in L^2_0(X,\mu).
$$
Observe that  we have $\Vert \kappa_0(\nu) \Vert \leq 1$ 
and hence  $r(\kappa_0(\nu)) \leq 1$  for the  spectral 
radius $r(\kappa_0(\nu))$ of  $\kappa_0(\mu)$.
Assume that  $\nu$ is aperiodic, that is, the support of  $\nu$ is not contained in the coset of a  proper subgroup of $\Ga$. Then the action of $\Ga$ on $X$ has a spectral gap
if and only if $r(\kappa_0(\nu))<1$ and this is 
 equivalent to  $\Vert \kappa_0(\nu) \Vert<1$; for more details, see the survey \cite{BekkaSG}.

In this paper, we will be concerned with the case where $X$ is an $S$-adic nilmanifold, to be introduced below, 
and $\Ga$ is a subgroup of automorphisms of $X.$

%\begin{abstract}
%\end{abstract}
 Fix a finite set  $\{p_1, \dots, p_r\}$ of integer primes and set $S= \{p_1, \dots, p_r, \infty\}.$
 For an integer $d\geq1,$  the product
 $$\QQ_S:= \prod_{p\in S} \QQ_p= \QQ_\infty\times\QQ_{p_1}\times\cdots\times \QQ_{p_r}$$
is a   locally compact ring, where $\QQ_\infty= \RRR$ and $\QQ_p$ is the field of $p$-adic numbers for a prime $p$. Let $\ZZ[1/S]=\ZZ[1/p_1, \cdots, 1/p_r]$ denote the subring  of $\QQ$ generated by $1$ and $\{1/p_1, \dots, 1/p_r\}.$ 
Through the diagonal embedding 
$$
\ZZ[1/S] \to \QQ_S, \qquad b\mapsto (b,\cdots, b),
$$
we may identify $\ZZ[1/S]$ with a  discrete and cocompact subring of $ \QQ_S.$

If  $\mathbf{G}$ is a linear algebraic  group defined over $\QQ,$ we denote by  
$\mathbf{G}(R)$  the group of elements of $\mathbf{G}$ with coefficients in $R$ and determinant invertible in $R,$ 
for every subring $R$ of  an overfield of $\QQ.$

Let $\mathbf{U}$ be a linear algebraic unipotent group defined over $\QQ,$
that is, $\mathbf{U}$ is an algebraic subgroup of  the group   
%$\mathcal U_n$ 
of  $n\times n$  upper triangular unipotent matrices  for some $n\geq 1,$
The group  $\mathbf{U}(\QQ_S)$ is a locally compact group and  $\La:=\mathbf{U}(\ZZ[1/S])$ is a cocompact lattice in  $\mathbf{U}(\QQ_S)$. The corresponding  $S$-\textbf{adic compact nilmanifold}  
$$\X= \mathbf{U}(\QQ_S)/\mathbf{U}(\ZZ[1/S])$$
will be equipped with the unique translation-invariant probability measure 
$\mu$ on the Borel subsets of $\X$.

For $p\in S,$ let  $\Aut(\mathbf{U}(\QQ_p))$ 
be  the group  of  continuous automorphisms of $\mathbf{U}(\QQ_s).$
Set $$\Aut(\mathbf{U}(\QQ_S)):=\prod_{p\in S} \Aut(\mathbf{U}( \QQ_p))$$
and denote by  $\Aut (\X)$ the subgroup
$$\{g\in\Aut(\mathbf{U}(\QQ_S))\mid g(\La) =\La\}.$$
Every  $g\in \Aut(\X)$ acts on $\X$ 
preserving the probability measure $\mu.$

The abelian quotient group 
$$\overline{\mathbf{U}(\QQ_S)}:=\mathbf{U}(\QQ_S)/[\mathbf{U}(\QQ_S), \mathbf{U}(\QQ_S)]$$ can be identified with $\QQ_S^d$ for some $d\geq 1$ and
the image $\Delta$ of $\mathbf{U}(\ZZ[1/S])$ in $\overline{\mathbf{U}(\QQ_S)}$ is a cocompact and discrete
subgroup of $\overline{\mathbf{U}(\QQ_S)};$
so, 
$$\Sol:= \overline{\mathbf{U}(\QQ_S)}/\Delta$$ is a  solenoid (that is,  is a finite-dimensional, connected, compact abelian group; see \cite[\S 25]{HeRo--63}). We refer to $\Sol$ as the $S$-\textbf{adic solenoid} attached to
the $S$- adic nilmanifold $\X.$
We equip $\Sol$ with the probability measure $\nu$ which is the image of $\mu$
under the canonical projection $\X\to \Sol.$

Observe that $\Aut(\QQ_S^d)$ is canonically isomorphic to $\prod_{s\in S}GL_d(\QQ_{s})$ and that  $\Aut(\Sol)$ can be identified with  the subgroup  $GL_d(\ZZ[1/S]).$
The group $\Aut(\X)$ acts naturally by automorphisms of $\Sol$;
we denote by 
$$p_S: \Aut(\X)\to GL_d(\ZZ[1/S])\subset  GL_d(\QQ)$$ 
the corresponding homomorphism.

\begin{theorem}
\label{Theo1}
Let $\mathbf{U}$ be an algebraic unipotent group defined over $\QQ$
and $S= \{p_1, \dots, p_r, \infty\},$ where $p_1, \dots, p_r$ are integer primes.
Let $\X= \mathbf{U}(\QQ_S)/\mathbf{U}(\ZZ[1/S]$ 
be the associated $S$-adic nilmanifold and let $\Sol$ be the corresponding $S$-adic solenoid, 
equipped with respectively  the probability measures  $\mu$ and $\nu$ as above.
 Let  $\Ga$ be a countable subgroup $\Aut(\X)$.
 The following properties are equivalent:
\begin{itemize}
 \item [(i)] The action $\Ga\curvearrowright  (\X,\mu)$  has a spectral gap.
 \item[(ii)] The action $p_S(\Ga) \curvearrowright  (\Sol, \nu)$  has a spectral gap,
where $p_S: \Aut(\X)\to GL_d(\ZZ[1/S])$ is the canonical  homomorphism.

\end{itemize}
\end{theorem}

Actions by groups of automorphisms (or more generally groups of by affine transformations) on 
the $S$-adic solenoid $\Sol$ have been completely characterized in \cite[Theorem 5]{BekkaFrancini}.
The following result is an immediate consequence of this characterization and of Theorem~ \ref{Theo1}.
For a subset $T$ of $GL_d(\KK)$ for a field $\KK,$ we denote by $T^t=\{g^t \mid g\in T\}$ the set 
of transposed matrices from $T$.
\begin{corollary}
\label{Cor1} 
With the notation as in Theorem~ \ref{Theo1}, the  following properties are equivalent:
\begin{itemize}
 \item [(i)] The action of $\Ga$ on the $S$-adic nilmanifold $\X$  does not have a spectral gap.
\item [(ii)] There exists a  non-zero linear subspace $W$  of $\QQ^d$ which is invariant 
under $p_S(\Ga)^t$ and such that the image of $p_S(\Ga)^t$ in $GL(W)$ is a virtually abelian group.%
%where $$p_S: \Aut(\X)\to \Aut(\Sol) =GL_d(\ZZ[1/S])$$
 %is the natural homomorphism.
\end{itemize}
\end{corollary}
Here is an immediate consequence of Corollary~\ref{Cor1}.
\begin{corollary}
\label{Cor2} 
With the notation as in Corollary~ \ref{Theo1},
assume that the linear representation of $p_S(\Ga)^t$ in $\QQ^d$
is irreducible. Then  the action $\Ga\curvearrowright  (\X,\mu)$   has a spectral gap.
\end{corollary}

 Recall that the action of a  countable group $\Ga$
 by measure preserving transformations on a probability space $(X, \mu)$ is  \textbf{strongly ergodic} 
 (see \cite{Schmidt}) if every
 sequence $(B_n)_n$ of measurable subsets of  $X$ which is asymptotically invariant
 (that is,  which is such that  $\lim_n\mu(\ga B_n \bigtriangleup B_n)=0$ for all $\ga\in \Ga$)
is trivial (that is, $\lim_n\mu( B_n)(1-\mu(B_n))=0$).
It is straightforward to check that  the spectral gap property implies strong ergodicity and  it is known that the converse 
does not hold in general.

The following corollary  is  a direct consequence of Theorem~\ref{Theo1} (compare with 
Corollary 2 in  \cite{BachirYves}).
\begin{corollary}
\label{Cor2}
With the notation as in \ref{Theo1},  the following properties are equivalent:
\begin{itemize} 
\item[(i)] The action $\Ga\curvearrowright  (\X,\mu)$
 has the spectral gap property. 
 \item[(ii)]  The action $\Ga\curvearrowright  (\X,\mu)$ is strongly ergodic. 
 \end{itemize}
\end{corollary}

 Theorem \ref{Theo1}  generalizes our previous work \cite{BachirYves}, 
 where  we treated  the real case (that is,   the case $S=\infty$).  We had to extend 
 our methods to the $S$-adic setting.  There are three main tools we  use in the proof:
 \begin{itemize}
 \item  a canonical  decomposition  of the Koopman representation of $\Ga$ in $L^2(\X)$ as
 a direct sum of certain  representations of $\Ga$ induced from stabilizers of 
 representations of $\mathbf{U}(\QQ_S)$ (see Proposition~\ref{Pro-KoopmanDec});
 \item  a result of Howe and Moore \cite{HoMo} about the decay of matrix coefficients of algebraic groups (see 
 Proposition~\ref{Pro-DecayRepAlgGr});
 \item the fact  that the irreducible representations of  $\mathbf{U}(\QQ_S)$ 
 appearing in the decomposition of $L^2(\X)$ are rational, in the sense that 
 the Kirillov data associated to each one of them are defined over $\QQ$ (see 
 Proposition~\ref{Prop-Rational}). 
 \end{itemize} 
 Another tool we constantly use is a generalized version 
 of Herz's majoration principle (see Lemma~\ref{Lem-Herz}).
 
 Given a probability measure $\nu$ on $\Ga,$  our approach does not seem 
 to provide  quantitative estimates  for the  operator norm of the convolution 
 operator $\kappa_0(\nu)$ acting on  $L^2_0(\X,\mu)$ 
 for a  general unipotent group $\mathbf{U}.$  However, 
 using known bounds for the so-called metaplectic representation of $Sp_{2n}(\QQ_p)$,
 we give such estimates in the case of $S$-adic Heisenberg nilmanifolds (see Section~\ref{S:Exa}).
 %In particular, we obtain the following result in dimension 3.
\begin{corollary}
 \label{Cor3}
For an integer $n\geq 1,$ let $\mathbf{U}=\mathbf{H}_{2n+1}$
be the $(2n+1)$- dimensional Heisenberg group
and  $\X=\mathbf{H}_{2n+1}(\QQ_S)/ \mathbf{H}_{2n+1}(\ZZ[1/S]).$  
Let $\nu$ be a probability measure on 
$$Sp_{2n}(\ZZ[1/S])\subset \Aut(\X).$$
Then 
$$
\Vert \kappa_0(\nu)\Vert \leq \max\{\Vert \la_\Ga(\nu)\Vert^{1/2n+2}, \Vert \kappa_1(\nu)\Vert\},
$$
where $\kappa_1$ is the restriction of $\kappa_0$ to $L^2_0(\Sol)$
%$$
%\Vert \kappa_0(\nu)\Vert \leq \Vert \la_\Ga(\nu)\Vert^{1/6},
%$$
and  $\la_\Ga$ is the regular representation
of the group $\Ga$ generated by the support of $\nu.$
In particular, in the case $n=1,$  the action of $\Ga$ on $\X$ has a spectral gap 
 if and only if $\Ga$
is non-amenable.
\end{corollary}

 \section{Extension of representations}
\label{S-Extension}
Let  $G$ be a  locally compact  group which we assume to be second countable.
We will need the notion of a projective  representation.
   Recall that a  mapping $\pi: G \to U(\H)$
from $G$ to the unitary group   of the Hilbert space $\H$
is  a \textbf{projective representation} of $G$
if the following holds:
\begin{itemize}
 \item $\pi(e)=I$,
\item for all $g_1,g_2\in G,$ there exists  $c(g_1 , g_2 )\in\mathbf C $
such that  $$\pi(g_1 g_2 ) = c(g_1 , g_2 )\pi(g_1 )\pi(g_2 ),$$
\item the function $g\mapsto \langle\pi(g)\xi,\eta\rangle$ is measurable for all
$\xi,\eta\in\H.$
\end{itemize}
The mapping  $c:G \times G \to {\mathbf S}^1$ is
a $2$-cocycle
with values in the unit cercle ${\mathbf S}^1.$
%The projective kernel of $\pi$ is defined
%in the same way as for an ordinary representation.
Every projective unitary representation of $G$
 can be  lifted to 
an ordinary unitary representation of a central extension
of $G $ (for all this, see \cite{Mackey-Livre} or \cite{Mackey}).

Let $N$ be a closed normal subgroup of $G$. Let
 $\pi$ be an irreducible unitary representation of $N$
 on a Hilbert space $\H.$
Consider the stabilizer 
$$
G_{\pi}=\{ g\in G\mid \pi^g   \text{ is equivalent to } \pi\}
$$
of $\pi$ in $G$ for the natural action of $G$ on  the unitary dual $\widehat{N}$ 
given by $\pi^g(n)= \pi(g^{-1}n g).$
Then $G_\pi$ is a closed subgroup of $G$
containing $N.$
The following lemma  is a well-known  part of Mackey's theory
of unitary representations of group extensions. 
\begin{lemma}
\label{Lem-ExtRep}
Let $\pi$ be an irreducible unitary representation of $N$ on 
the Hilbert space $\H.$  There exists a projective unitary representation $\widetilde\pi$
of ${G}_{\pi}$ on $\H$ which extends $\pi$. 
Moreover, $ \widetilde\pi$ is unique, up to scalars: 
any other projective unitary representation $\widetilde\pi '$ of ${ G}_\pi$ 
extending $\pi$ is  the form $\widetilde\pi '=\la\widetilde\pi$
 for a measurable function $\la: G_{\pi}\to {\mathbf S}^1.$
\end{lemma}
\begin{proof}
For every $g\in {G}_\pi$,  there exists a unitary
operator $\widetilde\pi(g)$ on $\H$ such that
$$
\pi (g(n))= \widetilde\pi(g) \pi(n) \widetilde\pi(g)^{-1} \tout n\in N.
$$
One can choose  $\widetilde\pi(g)$  
such that $g\mapsto \widetilde\pi(g)$ is a projective 
unitary representation of ${ G}_\pi$ which extends $\pi$ (see Theorem~8.2 in \cite{Mackey}).
The uniqueness of  $\pi$ follows from the irreducibility of $\pi$ and Schur's lemma. $\bsq$
\end{proof} 
 
\section{A weak containment result for induced representations }
\label{InducedRep}
Let $G$ be a locally compact group with Haar measure $\mu_G.$
Recall that a unitary representation $(\rho, \K)$ of $G$ 
is weakly contained in another unitary representation $(\pi, \H)$ of $G,$ 
if every matrix coefficient $g\mapsto \langle \rho(g)\eta\mid \eta\rangle$
of $\pi$ (for $\eta \in \K$) is the limit, uniformly of compact subsets of $G,$ of a finite sum of
matrix coefficients of $\pi.$
Equivalently, if 
$\Vert \rho(f)\Vert \leq \Vert \pi(f)\Vert$ for every $f\in C_c(G),$
where $C_c(G)$ is the space of continuous functions with compact support on $G$
and where the operator $\pi(f)\in \B(\H)$ is defined
by  the integral 
$$
\pi(f) \xi= \int_G f(g)\pi(g) d\mu_G(g) \tout \xi\in \H.
$$
The trivial  representation $1_G$ is weakly contained in $\pi$
if and only if there exists, for every compact subset $Q$ of $G$
and every $\varepsilon>0,$ there exists a unit vector $\xi\in \H$ 
which is $(Q,\varepsilon)$-invariant, that is such that
$$\sup_{g\in Q} \Vert \pi(g)\xi- \xi\Vert\leq \eps.$$

Let $H$ be a closed subgroup of $G.$
We will always assume that the coset space $H\bs G$
admits  a   non-zero $G$-invariant (possibly infinite) measure   on its Borel subsets.
Let $(\sigma,\K)$ be a unitary representation
of $H.$ We will use the following model
for the induced representation $\pi:=\ind_H^G \sigma$.
Choose a Borel fundamental domain $X\subset G$ for 
the action of $G$ on   $H\bs G$.
For $x \in X$ and $g\in G,$ let $x\cdot g\in X$ and $c(x,g)\in H$ 
%$X\times U_p \to X, (x, u) \mapsto u$ and $ X\times {U_p}\to M$ and $X\times U_p \to X, (x, u) \mapsto u$
be  defined by 
$$xg= c(x,g) (x\cdot g).$$
There exists a non-zero
$G$-invariant  measure on $X$ for the action $(x,g)\mapsto x\cdot g$ of $G$ on $X.$
The Hilbert space of $\pi$ is 
the space $L^2(X, \K, \mu)$  of 
all square-integrable measurable mappings $\xi: X\to \K$ and the 
 action of $G$ on $L^2(X, \K, \mu)$ is given by
$$
(\pi(g) \xi)(x)= \sigma (c(x,g))(\xi(x\cdot g)),\qquad g\in G,\ \xi\in L^2(X, \K, \mu),\ x\in X.
$$ 
Observe that, in the case where $\sigma$ is the trivial representation $1_H,$
the induced representation $\ind_H^G1_H$ is equivalent to  
\textbf{quasi-regular representation} $\lambda_{H\bs G},$
that is the natural representation of $G$ on $L^2(H\bs G,\mu)$ given by right translations.

We will use several times the following elementary but crucial lemma,
which can be viewed as a generalization of Herz's  majoration principle (see Proposition 17 in \cite{BachirYves}).
\begin{lemma}
\label{Lem-Herz}
Let $(H_i)_{i\in I}$ be a family of closed subgroups 
of $G$ such that $H_i\bs G$
admits  a   non-zero $G$-invariant  measure.
Let  $(\sigma_i,\K_i)$  be a unitary representation
of $H_i.$ Assume that $1_G$ is weakly contained in 
the direct sum $\oplus_{i\in I} \Ind_{H_i}^G\sigma_i.$
Then $1_G$ is weakly contained in 
$\oplus_{i\in I} \lambda_{H_i\bs G}.$
\end{lemma}
\begin{proof}
Let $Q$ be a compact subset of $G$ and $\eps.$
For every  $i\in I,$ let $X_i\subset G$
be a Borel fundamental domain for 
the action of $G$ on   $H_i\bs G$ and $\mu_i$
a    non-zero $G$-invariant  measure on $X_i.$
There exists a family  of vectors  $\xi_i\in L^2(X_i, \K_i, \mu_i)$ such that 
$ \sum_{i}\Vert \xi_i\Vert^2=1$ and
$$\sup_{g\in Q} \sum_{i}\Vert \Ind_{H_i}^G\sigma_i(g)\xi_i- \xi_i\Vert^2\leq \eps,$$

Define $\vfi_i$ in  $L^2(X_i, \mu_i)$
by $\vfi_i(x)=\Vert\xi_i(x)\Vert$.
Then $ \sum_{i}\Vert \vfi_i\Vert^2=1$ and, for every $g\in Q,$ we have
 \begin{align*}
\Vert \left(\Ind_{H_i}^G\sigma_i(g)\right)\xi_i- \xi_i\Vert^2&=
\int_{X_i} \Vert \sigma (c_i(x,g))(\xi_i(x\cdot_i g))- \xi_i(x)\Vert^2 d\mu_i(x)\\
&\geq \int_{X_i}\left\vert \Vert \sigma_i (c_i(x,g))(\xi_i(x\cdot_i g))\Vert  -\Vert \xi_i(x)\Vert \right\vert^2 d\mu_i(x)\\
&=\int_{X_i}\left| \Vert \xi_i(x\cdot_i g)\Vert  -\Vert \xi_i(x)\Vert \right|^2 d\mu_i(x)\\
&=\int_{X_i}\left| \vfi_i(x\cdot_i g)-\vfi(x) \right|^2 d\mu_i(x)\\
&=\Vert \lambda_{{H_i}\bs G}(g)\vfi_i- \vfi_i\Vert^2
\end{align*}
and the claim follows. $\bsq$
\end{proof}
\section{Decay of matrix coefficients of unitary representations}
\label{S:DecayMatrixCoeff}
We  recall a few general facts about  decay of matrix coefficients of  unitary representations,
Recall that  the  projective kernel of   a (genuine or projective) representation $\pi$ of the 
locally compact group $G$ is the closed normal subgroup $P_\pi$ of $G$ consisting
of the elements $g\in G$ such that $\pi(g)$ is a scalar multiple of the identity operator,
that is such that $\pi(g)=\lambda_\pi(g) I$ for some $\lambda_\pi(g)\in {\mathbf S}^1.$

  Observe also that, for $\xi, \eta\in\H,$ the absolute value
of the matrix coefficient
$$C^{\pi}_{\xi,\eta}: g\mapsto \langle \pi(g)\xi,\eta\rangle$$
is constant on cosets modulo $P_\pi.$  
For a real number  $p$ with $1\leq p <+\infty,$ the representation $\pi$ is said to be 
\textbf{strongly $L^p$  modulo $P_\pi$},
if there is dense subspace $D\subset \H$ 
such that $|C^{\pi}_{\xi,\eta}|\in L^p(G/P_\pi)$ for all $\xi,\eta\in D.$  
\begin{proposition}
\label{Pro-LpMatrCoeff}
Assume that the unitary representation  $\pi$  of the locally  compact group $G$
is strongly $L^p$ modulo $P_\pi$ for  $1\leq p <+\infty.$
Let $k$ be an  integer $k\geq p/2.$
 Then  the tensor power $\pi^{\otimes k} $ is contained in
 an infinite multiple of $\ind_{P_\pi}^G \lambda_\pi^k$,
 where $\lambda_\pi$ is the unitary character of $P_\pi$ associated to $\pi.$
\end{proposition}
  
\begin{proof} 
Observe  $\sigma:=\pi^{\otimes k}$ is 
square-integrable modulo $P_\pi$ for every integer $k\geq p/2.$
It follows (see Proposition 4.2 in \cite{HoMo} or  Proposition~1.2.3 in Chapter V of \cite{HoTa})
that $\sigma$  is contained in
 an infinite multiple of $\ind_{P_\sigma}^G \lambda_\sigma=\ind_{P_\pi}^G \lambda_\pi^k$.
 $\bsq$
 
\end{proof}

\section{The Koopman representation of the automorphism group of a homogeneous space}
\label{S-Kopman}

We establish  a decomposition result for the Koopman representation of a group of automorphisms of
an $S$-adic compact nilmanifold. We will state the result in the general context of a
 compact homogeneous space.
 
Let $G$ be a  locally compact  group and $\La$ a lattice in $G.$ We assume that $\La$
is cocompact in $G.$
%$H$  a cocompact subgroup in $G$ such that the restriction of the modular 
%function of $G$ to $H$ coincides with the modular  function of $G$
%(this is the case for instance if 
The homogeneous space $X:= G/\La$  carries a   probability
measure  $\mu$ on  the Borel subsets $X$ which is invariant by translations
with elements from $G.$
Every element from $$\Aut (X):=
\{\ga\in \Aut(G)\mid \ga(\La) \subset \La\}
$$
induces a Borel isomorphism of $X$, which  leaves $\nu$  invariant,
as follows from the uniqueness of $\nu$. 

%Let $\Ga$ be a subgroup  of  $\Aut(X).$ 
Given a subgroup $\Ga$ of  $\Aut(X),$ the following crucial proposition gives a decomposition of 
the associated Koopman  $\Ga$ on $L^2(X,\mu)$ as direct sum of
certain induced representations of $\Ga$.

\begin{proposition}
\label{Pro-KoopmanDec}
Let $G$ be a  locally compact  group and $\La$ a cocompact lattice in $G,$
and let $\Ga$ be a countable subgroup of $\Aut(X)$ for $X:= G/\La$
Let $\kappa$ be the  Koopman representation of  $\Ga$ 
associated to the action $\Ga\curvearrowright X.$
There exists a family $(\pi_i)_{i\in I}$ of irreducible unitary representations of $G$
such that $\kappa$ is  equivalent to a direct sum
$$\bigoplus_{i\in I} \ind_{\Ga_i}^{\Ga}(\widetilde{\pi_i}|_{\Ga_i}\otimes W_i),$$
where  $ \widetilde{\pi}_i$ is an  irreducible projective representation 
 of the stabilizer  $G_i$ of $\pi_i$ in $\Aut(G)\ltimes G$ extending $\pi_i$,
 and where  $W_i$ is a finite dimensional  projective unitary representation
of $\Ga_i :=\Ga\cap G_i$.
 \end{proposition}
 \begin{proof}
 We extend $\kappa$ to a  unitary representation, again denoted by  $\kappa,$ of $\Ga\ltimes G$ 
 on $L^2(X,\mu)$ given by
$$
\kappa(\ga, g) \xi(x)=\xi(\ga^{-1}(g x))  \tout\ga\in\Ga, g\in G, \xi\in L^2(X,\mu), \ x\in X.
$$
Identifying $\Ga$ and $G$ with subgroups of $\Ga\ltimes G$,  we have 
$$
 \kappa({\ga^{-1}}) \kappa(g) \kappa({\ga})= \kappa({\ga^{-1}(g)}) \tout \ga\in \Ga,\ n\in N.\leqno{(*)}
$$
Since $\La$ is cocompact in $G,$ we   can consider the decomposition of 
 $L^2(X,\mu)$  into  $G$--isotypical components:   we  have  (see Theorem in Chap. I, \S 3 of \cite{GGPS69})
$$
L^2(X,\mu)=\bigoplus_{\pi\in\Sigma} \H_{\pi},
$$
where $\Sigma$ is a certain set of  pairwise non-equivalent irreducible unitary representations of $G;$
for every $\pi\in\Sigma$, the space $\H_{\pi}$
is the union of the closed $\kappa(G)$-invariant subspaces $\K$
of $\H$ for which the corresponding representation of 
$G$ in ${\K}$   is  equivalent to $\pi$;
%According to \cite[Corollary2]{Moore}, every $\pi\in\Sigma$ is rational in the sense of Section~\ref{S5}.
moreover, the multiplicity of every $\pi$ is finite, 
that is,  every $\H_{\pi}$ is a direct sum of finitely many irreducible unitary  representations of $G$.

Let $\ga$ be a fixed automorphism in $\Ga.$ 
Let $\kappa^\ga$ be the conjugate representation of $\kappa$ by $\ga,$ that is,
$ \kappa^\ga(g)=\kappa(\ga g\ga^{-1})$ for all $g\in \Ga\ltimes N$.
On the one hand, for every $\pi\in\Sigma,$  the isotypical component of  $\kappa^\ga\vert _{G}$ corresponding to $\pi$
is $\H_{\pi^{\ga^{-1}}}$.  
On the other hand,  relation $(*)$ shows that
$\kappa(\ga)$ is a unitary equivalence between
 $\kappa\vert_{G} $ and  $\kappa^\ga\vert _{G}.$ 
  It follows that 
 $$
\kappa({\ga}) (\H_{\pi}) = \H_{\pi^{\ga}}\tout \ga\in \Ga;
$$
so, $\Ga$ permutes the $\H_\pi$'s among themselves according to its action on $\widehat{G}.$

Write $\Sigma=\bigcup_{i\in I} \Sigma_i,$ where the  $\Sigma_i$'s are   the $\Ga$-orbits
 in $\Sigma,$ and set
$$
\H_{\Sigma_i}=\bigoplus_{\pi\in\Sigma_i} \H_{\pi}.
$$
Every $\H_{\Sigma_i}$ is invariant under $\Ga\ltimes G$
and we have  an orthogonal decomposition  
$$
\H= \bigoplus_{i} \H_{\Sigma_i}.
$$
Fix $i\in I.$ Choose  a representation
${\pi}_i $ in $\Sigma_i$ and set $\H_i= \H_{\pi_i}.$
Let  $\Ga_i$  denote the stabilizer of $\pi_i$  in $\Ga.$ 
The space  $\H_i$
is invariant under  $\Ga_i.$ 
Let $V_i$ be  the corresponding representation
of $ \Ga_i$  on $\H_i.$

Choose   a set $S_i$ of representatives for the cosets in 
$$\Ga/\Ga_i= (\Ga\ltimes G)/ (\Ga_i\ltimes G)$$
 with $e\in S_i.$
Then $\Sigma_i=\{ \pi_i^s\ :\ s\in S_i\}$ 
and the Hilbert space $\H_{\Sigma_i}$
is the sum of mutually orthogonal spaces:
$$
\H_{\Sigma_i}=
\bigoplus_{s\in S_i}\H_i^s.
$$
Moreover, $\H_i^s$
is the image under $\kappa (s)$
of $\H_i$
for every $s\in S_i.$ 
This  means  that  the restriction   $\kappa_i$ of $\kappa$
to  $\H_{\Sigma_i}$ of the Koopman representation $\kappa$  of $\Ga$  
is equivalent to the induced representation 
$\ind_{\Ga_i}^{\Ga} {V_i}.$ 

Since every $\H_{i}$ is a direct sum of finitely many irreducible unitary  representations
of $G,$ we can assume that
$\H_i$ is the tensor product 
$$
\H_i  =\K_i\otimes \L_i
$$
of the Hilbert space $\K_i$ of $\pi_i$ with
a finite dimensional Hilbert space $\L_i,$ 
in such a way that
$$
V_i(g)= \pi_{i}(g) \otimes I_{\L_i} \tout g\in G.\leqno{(**)}
$$
Let  $\ga\in \Ga_i.$ By $(*)$ and $(**)$ above,  we have 
 $$
 V_i(\ga) \left(\pi_{i}(g) \otimes I_{\L_i}\right)V_i(\ga) ^{-1}  = \pi_{i}(\ga g\ga^{-1}) \otimes I_{\L_i} \tout  \ g\in G.\leqno{(***)}
$$
On the other hand, let  $G_i$ be the stabilizer of $\pi_i$ in $\Aut(G)\ltimes G$; then $\pi_i$ 
extends to an  irreducible projective representation $ \widetilde{\pi}_i$
 of $G_i$  (see Section \ref{S-Extension}).
Since 
$$
 \widetilde{\pi_i}(\ga)  \pi_{i}(g)\ \widetilde{\pi_i}(\ga^{-1})= \pi_{i}(\ga g\ga^{-1}) \tout  g\in G, 
 $$
it follows from $(***)$  that 
  $\left(\widetilde{\pi_i}(\ga^{-1})\otimes I_{\L_i}\right)V_i(\ga)$
commutes with $\pi_i(g)\otimes I_{\L_i}$ for all $g\in \widetilde{G}_i.$ 
 As $\pi_i$ is irreducible,  there exists a unitary operator
$W_i(\ga)$ on $\L_i$ such that 
$$
V_i(\ga)= \widetilde{\pi_i}(\ga)\otimes W_i(\ga).
$$
 It is clear that $W_i$ is a projective unitary representation
of $\Ga_i\ltimes G$,  since  $V_i$  is a unitary representation
of $\Ga_i\ltimes G$.$\bsq$
\end{proof}

\section{Unitary dual  of solenoids}
\label{UniDualSolenoid}
Let $p$ be either a prime integer or $p=\infty.$
Define an element $e_p$ in the unitary dual  group $\widehat{\QQ_p}$ of the additive group of $\QQ_p$
by  $e_p(x)= e^{2\pi i x}$ if $p=\infty$ 
and $e_p(x)= \exp(2\pi i  \{x\}),$ where $\{x\}= \sum_{j=m}^{-1} a_j p^j $ denotes the ``fractional part" of a 
 $p$-adic number $x= \sum_{j=m}^\infty a_j p^j$ for integers $m\in \ZZ$ and $a_j \in \{0, \dots, p-1\}$.
 Observe that  $\Ker (e_p)=\ZZ$ in case $p=\infty$ and 
 that $\Ker (e_p)=\ZZ_p$ in case $p$ is a prime integer, where $\ZZ_p$ is the ring of $p$-adic integers.
 The map 
 $$ \QQ_p \to \widehat{\QQ_p}, \qquad y\mapsto (x\mapsto e_p(xy))$$
 is an isomorphism of topological groups  (see \cite[Section D.4]{BHV}).

Fix an integer $d\geq 1.$
Then $\widehat{\QQ_p^d}$ will be identified with $\QQ_p^d$
by means of the map
$$\QQ^d \to \widehat{\QQ_p^d}, \qquad y\mapsto x\mapsto e_p( x\cdot y),$$
where $x\cdot y= \sum_{i=1}^d x_i y_i$ for $x=(x_1, \dots, x_d), 
y=(y_1, \dots, y_d)\in \QQ_p^d.$
%the  linear dual space $V^*$
%Let $V$ be a finite dimensional vector space over $\QQ_p$.
%The dual group of $\widehat{V}$ will be identified with 
%the  linear dual space $V^*$  through the map 
%$$V^* \to \whitehat{V}, f\mapsto x\mapsto e( f(x)).$$

Let $S= \{p_1, \dots, p_r, \infty\}$, where $p_1, \dots, p_r$ are integer primes.
For an integer  $d\geq 1$, consider the  \textbf{$S$-adic solenoid} 
$$\Sol=\QQ_S^d/\ZZ[1/p_1, \cdots, 1/p_r]^d,$$
where  $\ZZ[1/p_1, \cdots, 1/p_r]^d$ is embedded diagonally in 
$\QQ_S= \prod_{p\in S} \QQ_p.$
 
Then  $\widehat{\Sol}$ is identified  with the annihilator of $\ZZ[1/p_1, \cdots, 1/p_r]^d$ in
$\QQ_S^d,$ that is, with $\ZZ[1/p_1, \cdots, 1/p_r]^d$ embedded in 
$\QQ_S^d$ via the map
$$
\ZZ[1/p_1, \cdots, 1/p_r]^d\to \QQ_S^d, \qquad b\mapsto (b, -b \cdots,- b).
$$
Under this identification, the dual action of the automorphism group 
$$
\Aut(\QQ_S^d)\cong  GL_d(\RRR)\times GL(\QQ_{p_1})\times\cdots\times GL(\QQ_{p_r})
$$
on  $\widehat{\QQ_S^d}$  corresponds to the right action on  $\RRR^d\times\QQ_{p_1}^d\times\cdots\times \QQ_{p_r}^d$ given by 
$$
((g_\infty, g_1, \cdots, g_r), (a_\infty, a_1, \cdots, a_r))\mapsto  (g_\infty^t a_\infty , g_1^t a_1, \cdots, g_r^t a_r),
$$
where $(g,a)\mapsto ga$ is the usual (left) linear action of $GL_d(\mathbf{k})$ on $\mathbf{k}^d$ for a field $\mathbf{k}.$

\section{Unitary representations of unipotent groups}
\label{UniRepUnipotent}
%\section{Some basic facts on Kirillov's theory and on decay  of matrix coefficients of unitary representations}
%\label{S2}
Let $\mathbf{U}$ be a linear algebraic unipotent group defined over $\QQ$.
The Lie algebra $\mathfrak{u}$ is defined over $\QQ$ and 
the exponential map $\exp: \mathfrak{u}\to U$ is a bijective morphism of algebraic varieties.

Let $p$ be either a prime integer or $p=\infty.$  The irreducible unitary representations of   $U_p:=\mathbf{U}(\QQ_p)$ are parametrized by Kirillov's theory as follows.

The Lie algebra of $U_p$ is  $\mathfrak{u}_p=\mathfrak{u}(\QQ)\otimes_{\mathbf Q} \mathbf \QQ_p,$
where $\mathfrak{u}(\QQ)$ is the Lie algebra over $\QQ$ consisting of the $\QQ$-points in 
$\mathfrak u.$

 Fix an element $f$ in the dual space ${\mathfrak u}_p^*= {\mathcal Hom}_{\QQ_p}({\mathfrak u}_p, \QQ_p)$ of $\mathfrak{u}_p.$
There exists a polarization $\mathfrak m$ for $f,$ that is, a Lie subalgebra $\mathfrak m$ 
of ${\mathfrak u}_p$ such that $f([{\mathfrak m},{\mathfrak m}])=0$ and which is of maximal dimension.
% the codimension of $\mathfrak m$ is $\frac{1}{2}\dim ({\rm Ad}^* (U) l)$, where ${\rm Ad}^* (U) l$  is the orbit of $l$ under the co-adjoint representation ${\rm Ad}^*$ of $N.$ 
The induced representation  
${\rm Ind}_M^{U_p} \chi_f$ is irreducible, where $M=\exp(\mathfrak m)$ and 
$\chi_f$ is the unitary character of $M$ defined by
$$\chi_f(\exp X)=e_p(f(X)) \tout X \in {\mathfrak m},$$
where $e_p\in \widehat{\QQ_p}$ is  as in Section~\ref{UniDualSolenoid}.
The unitary equivalence class of ${\rm Ind}_M^{U_p} \chi_f$  only depends  on the co-adjoint orbit
${\rm Ad}^* (U_p) f$ of $f$. The map
$$
{\mathfrak n}_p^*/{\rm Ad}^* (U_p)\to \widehat{U_p},\qquad {\mathcal O}\mapsto \pi_{\mathcal O}
$$
called the Kirillov map, 
from the orbit space ${\mathfrak u}_p^*/{\rm Ad}^*(U_p)$ of the co-adjoint representation
to  the   unitary dual $\widehat{U}_p$ of $U_p$, is  a bijection.
In particular, $U_p$ is a so-called type I locally compact group.
For all of this, see \cite{Kirillov} or \cite{CoGr} in the case $p=\infty$ and \cite{Moore}
in the case  of a prime integer $p$.

The group $\Aut (U_p)$ of continuous automorphisms of $U$ 
can be identified with
the group of $\QQ_p$-points of the algebraic group
$\Aut (\mathfrak u)$ of automorphisms of the Lie algebra
$\mathfrak u$  of $\mathbf{U}.$
Notice also that the natural action of $\Aut(U_p)$  on ${\mathfrak u}_p$ as well as its dual action
on ${\mathfrak u}_p^*$ are algebraic.

Let  $\pi\in \widehat{U_p}$ with corresponding Kirillov orbit $\mathcal{O}_\pi$ and $g\in \Aut(U_p).$
Then $g(\mathcal{O}_\pi)$ is the Kirillov orbit associated to the conjugate representation $\pi^g.$

 \begin{lemma}
\label{Lem-RationalStab}
 Let $\pi$ be an irreducible unitary representation of $U_p.$ 
 The stabilizer $G_\pi$ of $\pi$  in $\Aut(U_p)$  is an algebraic subgroup  of $\Aut(U_p).$
 \end{lemma}
 \begin{proof}
 Let ${\mathcal O}_\pi\subset {\mathfrak u}_p^*$ be the 
 Kirillov orbit corresponding to $\pi.$ Then 
 $G_\pi$ is the set of $g\in \Aut(U_p)$ such that 
 $g(\mathcal{O}_\pi)= \mathcal{O}_\pi.$
 As $\mathcal{O}_\pi$ is an algebraic subvariety of  ${\mathfrak u}_p^*$,
 the claim follows.
  $\bsq$
 \end{proof}
\section{Decay of matrix coefficients of unitary representations of $S$-adic groups}
Let $p$ be an integer prime or  $p=\infty$ and let
$\mathbf{U}$ be a linear algebraic unipotent group defined over $\QQ_p$.
Set $U_p:=\mathbf{U}(\QQ_p).$

Let $\pi$ be an irreducible unitary representation of
$U_p$. Recall (see Lemma~\ref{Lem-RationalStab})
 that the stabilizer $G_\pi$ of $\pi$  in $\Aut(U_p)$  is an algebraic subgroup  of $\Aut(U_p).$
 Recall  also (see Lemma~\ref{Lem-ExtRep}) that $\pi$ extends to a projective representation of $G_\pi$.
 The following result was proved in Proposition 22 of \cite{BachirYves}
 in the case where $p=\infty,$ using arguments from \cite{HoMo}.
 The proof in the case  where $p$ is a prime integer is  along similar lines
 and will be omitted.
\begin{proposition}
\label{Pro-DecayRepAlgGr} 
Let $\pi$ be an irreducible unitary representation of
$U_p$ and let $\widetilde{\pi}$
be a projective unitary representation  of ${G}_{\pi}$  which extends ${\pi}.$
There exists a real number  $r\geq 1,$ only depending  on the
dimension of $G_\pi,$ such that $\widetilde{\pi}$  is strongly $L^r$  modulo  its projective
kernel.
\end{proposition}

We will  need later a  precise description of the projective kernel of  a representation $\widetilde{\pi}$ as above.
\begin{lemma}
\label{Lem-ProjKernelInduced}
 Let $\pi$ be an irreducible unitary representation of $U_p$ and 
  $\widetilde{\pi}$
 a projective unitary representation  of ${G}_{\pi}$  which extends ${\pi}.$
 Let ${\mathcal O}_\pi\subset {\mathfrak u}_p^*$ be the corresponding 
 Kirillov orbit of $\pi.$ For $g\in \Aut(U_p),$ the following properties are equivalent:
 \begin{itemize}
 \item [(i)] $g$ belongs to the projective kernel $P_{\widetilde{\pi}}$ of  $\widetilde{\pi};$ 
  \item[(ii)]  for every $u\in U_p,$ we have 
  $$g(u)u^{-1}\in \bigcap_{f\in {\mathcal O}_\pi}\exp (\Ker(f)).$$ 
  \end{itemize}
  \end{lemma}

\begin{proof}
We can assume that $\pi= {\rm Ind}_M^{U_p} \chi_{f_0}$, for  $f_0\in {\mathcal O}_\pi,$
and $M= \exp \mathfrak m$ for a polarization $\mathfrak m$ of $f_0$.
For $g\in \Aut(U_p),$ we have $g\in P_{\widetilde{\pi}}$ if and only if
$$
\pi(g(u))= \pi(u) \tout u\in U_p
$$
that is, 
$$
g(u)u^{-1} \in \Ker (\pi) \tout u\in U_p.
$$
Now, we have (see \cite[Lemma 18]{BachirYves})
$$\Ker (\pi) =\bigcap_{f\in  {\mathcal O}_\pi} \Ker (\chi_{f})$$
and so, $g\in P_{\widetilde{\pi}}$ if and only if 
$$g(u)u^{-1}\in \bigcap_{f\in  {\mathcal O}_\pi} \Ker (\chi_{f}) \tout u\in U_p.$$

Let $g\in P_{\widetilde{\pi}}.$ 
Denote by $X\mapsto g(X)$ the automorphism 
of $\mathfrak{u}_p$ corresponding to $g.$ 
Let $u=\exp (X)$ for $X\in \mathfrak{u}_p$
and $f\in  {\mathcal O}_\pi$. Set $u_t= \exp(tX).$
By the Campbell Hausdorff formula, 
%we have
%$$\exp \left(t (g(X)-X) + \dfrac{t^2}{2}[g(X)-X, Y]+ \dfrac{t^3}{12}([g(X)-X, [g(X)-X,Y]]- [Y, [g(X)-X, Y]])+\cdots$$
 %where the 
there exists $Y_1, \dots Y_r\in \mathfrak{u}_p$ such that
$$
g(u_t)(u_t)^{-1}= \exp \left(t Y_1 + t^2Y_2+\dots+ t^r Y_r\right),
$$
for every $t\in \QQ_p;$ 
%more precisely, we have
% $$Y_1= -\dfrac{1}{2}[g(X), X]), Y_2= \dfrac{1}{12}(-[g(X), [g(X), X]]  -[X, [g(X), X]]), \dots. $$
Since
$$1=\chi_{f} (g(u_t)(u_t)^{-1})=e_p(f\left( t Y_1+ t^2Y_2+\dots+ t^r Y_2\right)),\leqno{(*)}$$
it  follows that the polynomial 
$$t\mapsto Q(t)= t f(Y_1) + t^2f(Y_2)+\dots+ t^r f(Y_r)$$
takes its values in $\ZZ$ in case $p=\infty$ and in $\ZZ_p$ (and so $Q$ has bounded image) otherwise.
This clearly implies that $Q(t)=0$ for all $t\in \QQ_p$; in particular, 
we have  
$$\log (g(u)u^{-1})= Y_1+Y_2+\cdots +Y_r \in \Ker (f).$$
 This shows that (i) implies (ii).

Conversely,  assume that (ii) holds. Then clearly 
$$g(u)u^{-1}\in \bigcap_{f\in  {\mathcal O}_\pi} \Ker (\chi_{f}) \tout u\in U_p$$
and so $g\in P_{\widetilde{\pi}}$. $\bsq$
 \end{proof}

%We will later use  the notion of an  irreducible unitary representations of $U$  which is  rational.
% the dual vector space ${\mathfrak u}^*$ 

%\begin{definition}
 %\label{Def2}
%An irreducible  unitary representation $\pi$ of $U$ is \emph{rational} if its co-adjoint orbit ${\mathcal O}_\pi$
%is rational, that is, if ${\mathcal O}_\pi\cap {\mathfrak u}^*_{\mathbf Q}\neq \emptyset.$
%\end{definition}

\section{Decomposition of the Koopman representation for a nilmanifold}
\label{S-Koopman Nilmanifold}
Let $\mathbf{U}$ be a linear algebraic unipotent group defined over $\QQ$.
Let  $S= \{p_1, \dots, p_r, \infty\}$, where $p_1, \dots, p_r$ are integer primes.
Set 
$$U:=\mathbf{U}(\QQ_S)=\prod_{s\in S} U_p.$$
Since $U$ is a type I group, the unitary dual $\hU$ of $U$  can be identified
with the cartesian product $\prod_{s\in S} \widehat{U_p}$ via the map
$$
\prod_{s\in S}  \widehat{U_p}\to \hU,  \qquad (\pi_p)_{p\in S} \mapsto\otimes_{s\in S} \pi_p,
$$
where $\otimes_{s\in S} \pi_p= \pi_{\infty}\otimes \pi_{p_1} \otimes \dots \otimes \pi_{r}$
is the tensor product of the $\pi_p$'s.

Let $\La:=\mathbf{U}(\ZZ[1/S])$ and consider the corresponding  $S$-\textbf{adic compact nilmanifold} 
$$\X:= U/\Lambda,$$
  equipped with the unique $U$-invariant probability measure $\mu$ on its Borel subsets.

The associated $S$-\textbf{adic solenoid} is 
$$
\Sol= \overline{U}/\overline{\La},
$$
%where $U$ denotes the closed commutator subgroup of $U,$
where $\overline{U}:=U/[U,U]$ is the quotient of $U$ by its closed commutator subgroup 
$[U,U]$ and where
 $\overline{\La}$ is the image of $\mathbf{U}(\ZZ[1/S])$ in $\overline{U}.$

Set 
$$\Aut(U):=\prod_{s\in S} \Aut(\mathbf{U}( \QQ_s))$$
and denote by  $\Aut (\X)$ the subgroup of all 
$g\in  \Aut(U)$ with  $g(\La) =\La.$

Let $\Ga$ be a  subgroup of  $\Aut (\X)$.
Let $\kappa$ be the  Koopman representation of  $\Ga\ltimes U$  on $L^2(\X)$
associated to the action $\Ga\ltimes U\curvearrowright \X.$
By Proposition~\ref{Pro-KoopmanDec}, there exists a family $(\pi_i)_{i\in I}$ of irreducible representations of $U,$
such that $\kappa$ is equivalent to 
$$\bigoplus_{i\in I} \ind_{\Ga_i\ltimes U}^{\Ga\ltimes U}(\widetilde{\pi_i}\otimes W_i),$$
where  $ \widetilde{\pi}_i$ is an  irreducible projective representation $ \widetilde{\pi}_i$
 of the stabilizer  $G_i$ of $\pi_i$ in $\Aut(U)\ltimes U$ extending $\pi_i$,
 and where  $W_i$ is a  projective unitary representation
of $G_i \cap (\Ga\ltimes U)$.

Fix  $i\in I.$ We  have  $\pi_i=\otimes_{p\in S}\pi_{i,p}$ for irreducible representations $\pi_{i, p}$ of $U_p.$

%Let $O_{i, p}\subset \mathfrak{u}_p^*$ be the associated coadjoint orbit by Kirillov's theory 
%(see Section~\ref{UniRepUnipotent}). 
We will need the following  more precise description of  $\pi_i.$  
Recall that  $\mathfrak u$ is the Lie algebra of $\mathbf U$
and that $\mathfrak{u}(\QQ)$ denotes the Lie algebra over $\QQ$ consisting of the $\QQ$-points in 
$\mathfrak u.$  Let  $\mathfrak{u}^*(\QQ)$ be the set of $\QQ$-rational points 
in the dual space $\mathfrak u^*;$  so,  $\mathfrak{u}^*(\QQ)$ is the subspace of
$f\in \mathfrak u^*$ with $f(X)\in \QQ$ for all $X\in \mathfrak{u}(\QQ).$
Observe that, for $f\in \mathfrak{u}^*(\QQ),$ we have
  $f(X)\in \QQ_p$ for all $X\in \mathfrak{u}_p=\mathfrak{u}(\QQ_p)$.
  
 A polarization for $f\in \mathfrak{u}^*(\QQ)$ is a Lie subalgebra $\mathfrak m$ 
of ${\mathfrak u}(\QQ)$ such that $f([{\mathfrak m},{\mathfrak m}])=0$ and which is of maximal dimension
with this property.

\begin{proposition}
\label{Prop-Rational}
Let  $\pi_i= \otimes_{p\in S} \pi_{i, p}$ be one of the irreducible representations of $U=\mathbf{U}(\QQ_S)$ appearing  
in the decomposition $L^2(\X)$ as above.
%\begin{itemize}
%\item[(i)] 
There exist $f_i\in \mathfrak{u}^*(\QQ)$
and  a polarization $\mathfrak{m}_{i} \subset \mathfrak{u}(\QQ)$ for $f_i$
with the following property: for every $p\in S,$ the representation $\pi_{i,p}$ is equivalent 
to $\Ind_{ M_{i,p}}^U \chi_{f_i},$ where $M_{i,p}=\exp(\mathfrak{m}_{i,p})$ and 
$\chi_{f_i}$ is the unitary character of $M_{i,p}$ given by
$$\chi_{f_i}(\exp X)=e_p( f_i(X)),\tout X \in {\mathfrak m}_{i,p}=\mathfrak{\mathfrak m}_{i}\otimes_{\mathbf Q} \mathbf \QQ_p,$$ where $e_p\in \widehat{\QQ_p}$  is  as in Section~\ref{UniDualSolenoid}.
%\item[(ii)] Let $f_i'\in \mathfrak{u}^*(\QQ)$ have the same property as in (i). Then 
%there exists $g\in \mathbf{U}(\QQ)$ such that  $f'= \Ad^*(g) f.$
%\end{itemize}
\end{proposition}
\begin{proof}
The same result is proved in Theorem 11 in \cite{Moore} (see also Theorem 1.2 in \cite{Fox})
for the Koopman representation of $\mathbf{U}(\mathbf{A})$ in $L^2(\mathbf{U}(\mathbf{A})/ \mathbf{U}(\mathbf{Q})),$
where $\mathbf{A}$ is the ring of adeles of $\QQ.$
We could check that the proof, which proceeds by induction of the dimension of $\mathbf{U}$, carries over to the Koopman representation on
$L^2(\mathbf{U}(\QQ_S)/ \mathbf{U}(\mathbf{\ZZ}[1/S])$, with the appropriate changes.
%Alternatively, we can deduce our claim from  the result for $\mathbf{U}(\mathbf{A})$, as we are going to show.
We prefer to deduce our claim from  the result for $\mathbf{U}(\mathbf{A})$, as follows.

It is well-known (see \cite{Weil}) that
$$\AA= \left(\QQ_S \times \prod_{p\notin S} \ZZ_p\right) +\QQ$$
and that 
$$\left(\QQ_S \times \prod_{p\notin S} \ZZ_p\right)\cap \QQ= \ZZ[1/S].
$$ 
 This gives rise to a well defined projection  $\vfi:\mathbf{A}/ \mathbf{Q} \to \QQ_S/\ZZ[1/S]$
  given by 
$$\vfi\left((a_S, (a_p)_{p\notin S}) +\QQ\right)= a_S+\ZZ[1/S] \tout a_S\in \QQ_S, (a_p)_{p\notin S}\in \prod_{p\notin \P} \ZZ_p;$$ 
so the fiber  over  a point $a_S+\ZZ[1/S]\in  \QQ_S/\ZZ[1/S]$ is
 $$\vfi^{-1}(a_S+\ZZ[1/S])= \{(a_S, (a_p)_{p\notin S}) +\QQ\mid a_p\in \ZZ_p \text{ for all } p\}.$$
 This induces an identification of 
 $\mathbf{U}(\mathbf{\QQ_S})/ \mathbf{U}(\mathbf{\ZZ}[1/S])=\X$
 with the double coset space $K_S\backslash \mathbf{U}(\mathbf{A})/ \mathbf{U}(\mathbf{Q}),$
 where $K_S$ is the compact subgroup 
 $$K_S=\prod_{p\notin S}\mathbf{U}(\ZZ_p)$$
 of $\mathbf{U}(\mathbf{A}).$
 Observe that this identification is equivariant under translation by elements from $\mathbf{U}(\QQ_S).$
 In this way, we can view 
$L^2(\X)$ as the $\mathbf{U}(\QQ_S)$-invariant subspace  $L^2(K_S\backslash \mathbf{U}(\mathbf{A})/ \mathbf{U}(\mathbf{Q}))$ of $L^2(\mathbf{U}(\mathbf{A})/ \mathbf{U}(\mathbf{Q})).$

Choose a system $T$ of representatives  for the $\Ad^*(\mathbf{U}(\QQ))$-orbits
in $ \mathfrak{u}^*(\QQ)$. By \cite[Theorem 11]{Moore}, for every $f\in T,$ we can find 
 a polarization $\mathfrak{m}_{f}\subset \mathfrak{u}(\QQ)$  for $f$ 
with the following property: setting 
$$\mathfrak{m}_{f}(\mathbf{A})=\mathfrak{\mathfrak m}_{f}\otimes_{\mathbf Q} \mathbf{A},$$
we have a decomposition 
$$
L^2(\mathbf{U}(\mathbf{A})/ \mathbf{U}(\mathbf{Q}))= \bigoplus_{f\in T} \H_f
$$
into irreducible  $\mathbf{U}(\mathbf{A})$-invariant subspaces $\H_f$ such that 
the representation $\pi_f$ of  $\mathbf{U}(\mathbf{A})$ in $\H_f$ is equivalent to 
 $\Ind_{ M_{f}(\mathbf{A})}^{\mathbf{U}(\mathbf{A}) }\chi_{f},$ where 
$$M_{f}(\mathbf{A})=\exp(\mathfrak{m}_{f}(\mathbf{A}))$$ and 
$\chi_{f, \mathbf{A}}$ is the unitary character of $M_{f}(\mathbf{A})$ given  by
$$\chi_{f,  \mathbf{A}}(\exp X)=e( f(X)),\tout X \in {\mathfrak m}_{f}(\mathbf{A});$$
here,  $e$ is the unitary character of $\mathbf{A}$  defined by 
$$
e((a_p)_p) = \prod_{p\in \P \cup \{\infty\}} e_p(a_p) \tout (a_p)_p\in \mathbf{A},
$$
where $\P$ is the set of integer primes.

We have 
$$L^2(K_S\backslash \mathbf{U}(\mathbf{A})/ \mathbf{U}(\mathbf{Q}))= \bigoplus_{f\in T} \H_f^{K_S},$$
where $ \H_f^{K_S}$ is the space of $K_S$-fixed vectors in $\H_f.$
It is clear that the representation of $\mathbf{U}(\QQ_S)$ in $\H_f^{K_S}$
is equivalent to 
$$\Ind_{ M_{f}(\mathbf{\QQ_S})}^{\mathbf{U}(\QQ_S) }(\otimes_{p\in S}\chi_{f, p})= 
\otimes_{p\in S}\left(\Ind_{ M_{f}(\mathbf{\QQ_p})}^{\mathbf{U}(\mathbf{\QQ_p}) }\chi_{f,p}\right),$$
where  $\chi_{f, p}$ is the unitary character of $M_{f}(\QQ_p)$ given  by
$$\chi_{f,  p}(\exp X)=e_p( f(X)),\tout X \in {\mathfrak m}_{f}(\QQ_p).$$
Since $ M_{f}(\mathbf{\QQ_p})$ is a polarization for $f$, each of  the $\mathbf{U}(\QQ_p)$-representations 
$\Ind_{ M_{f}(\mathbf{\QQ_p})}^{\mathbf{U}(\mathbf{\QQ_p}) }\chi_{f,p}$ and, hence, each of the
$\mathbf{U}(\QQ_S)$-representations  
$$\Ind_{ M_{f}(\mathbf{\QQ_S})}^{\mathbf{U}(\QQ_S) }(\otimes_{p\in S}\chi_{f, p})$$ is irreducible.
This proves the claim.
 $\bsq$
 \end{proof}

 We establish another crucial fact about the representations $\pi_i$'s in the following proposition.

 \begin{proposition}
 \label{Prop-Rational2} 
 With the notation of Proposition~\ref{Prop-Rational}, let 
 ${\mathcal O}_\QQ(f_i)$ be the co-adjoint orbit of $f_i$ under $\mathbf{U}(\QQ)$
 and set 
 $$
 \mathfrak{k}_{i,p}=\bigcap_{f\in  {\mathcal O}_\QQ(f_i)} \mathfrak{k}_p(f),
 $$
where  $\mathfrak{k}_{p}(f)$ is the kernel of $f$ in  $\mathfrak{u}_{p}.$
Let $K_{i,p}= \exp ( \mathfrak{k}_{i,p})$ and $K_i= \prod_{p\in S} K_{i,p}.$
 \begin{itemize}
  \item[(i)] $K_i$ is a closed normal subgroup of $U$
  and  $K_{i} \cap \La= K_i\cap \mathbf{U}(\ZZ[1/S])$ is a lattice in $K_i.$
  \item[(ii)]  Let $P_{\widetilde\pi_i}$ be the projective
  kernel of the extension $\widetilde \pi_i$ of  $\pi_i$ to the stabilizer 
  $G_i$ of $\pi$   in $\Aut(U)\ltimes U$. 
  For $g\in G_i$, we have $g\in P_{\widetilde\pi_i}$ if and only if 
 $g(u)\in u K_i$ for every  $u\in U.$
 % that is, if and only if $K_i$ is invariant under $P_i$ and  $ P_i$ acts as the identity on $U/K_i.$ 
  \end{itemize}
 \end{proposition}
 \begin{proof}
 (i) 
 %Observe that $\mathfrak{k}_{i,p}$ is an ideal in $\mathfrak{u}(\QQ_p),$ since it is $\Ad(\mathbf{U}(\QQ_p))$-invariant; so, 
 %$K_{i,p}= \exp ( \mathfrak{k}_{i,p})$ is a normal subgroup of $\mathbf{U}(\QQ_p)$.
Let 
 $$
 \mathfrak{k}_{i,\QQ}=\bigcap_{f\in  {\mathcal O}_\QQ(f_i)} \mathfrak{k}_\QQ(f),
 $$
 where  $\mathfrak{k}_{\QQ}(f)$ is the kernel of $f$ in  $\mathfrak{u}(\QQ).$
 Observe that $ \mathfrak{k}_{i,\QQ}$ is an ideal in $\mathfrak{u}(\QQ),$
 since it is $\Ad(\mathbf{U}(\QQ))$-invariant.
 So, we have 
 $$\mathfrak{k}_{i, \QQ}= \mathfrak{k}_{i}(\QQ)$$ for an ideal   $\mathfrak{k}_{i}$
 in $ \mathfrak{u}.$
 Since  $f\in \mathfrak{u}^*(\QQ)$ for  $f\in  {\mathcal O}_\QQ(f_i),$ we have 
 $$
  \mathfrak{k}_{i,p}(f)=\mathfrak{k}_{i,\QQ}(f)\otimes_{\mathbf Q} \QQ_p
 $$
 and hence
 $$
\mathfrak{k}_{i,p}=   \mathfrak{k}_{i}(\QQ_p).
$$
Let $\mathbf{K}_i=\log ( \mathfrak{k}_{i}).$ Then $\mathbf{K}_i$ is a normal
algebraic $\QQ$-subgroup  of $\mathbf{U}$ and we have 
$K_{i,p}=  \mathbf{K}_i(\QQ_p)$ for every $p$; so,
$$K_i= \prod_{s\in S} \mathbf{K}_i(\QQ_p)=  \mathbf{K}_i(\QQ_S)$$
and  $K_i \cap \La=  \mathbf{K}_i(\ZZ[1/S])$ is a lattice in $K_i.$
This proves Item  (i).

 To prove Item (ii),  observe that $$P_{\widetilde\pi_i}= \bigcap_{p\in S} P_{i,p},$$ where
 $P_{i,p}$ is the projective kernel  of $\widetilde{\pi_{i,p}}$.

 Fix $p\in S$ and let $g\in G_i$.
  By Lemma~\ref{Lem-ProjKernelInduced},  $g\in P_{i,p}$ if and only if 
% $g$ fixed every functional from the orbit ${\mathcal O}_{\QQ_p}(f_i)$ of $f_i$ under $U_p.$
 % Hence, $g\in P_i$ if and only if $g$ fixes every functional from ${\mathcal O}_\QQ(f_i)$;
  %equivalently, if and only if $g(u)\in u K_i$ for every  $u\in U.$
  $g(u)\in u K_{i,p}$ for every  $u\in U_p=\mathbf{U}(\QQ_p). $
  This finishes the proof. $\bsq$
  
 \end{proof}

\section{Proof of Theorem~\ref{Theo1}}

Let $\mathbf{U}$ be a linear algebraic unipotent group defined over $\QQ$
and $S= \{p_1, \dots, p_r, \infty\}$, where $p_1, \dots, p_r$ are integer primes. 
Set  $U:=\mathbf{U}(\QQ_S)$   and  $\La:=\mathbf{U}(\ZZ[1/S])$.
%Let   $\X:= \mathbf{U}(\QQ_S)/\mathbf{U}(\ZZ[1/S])$ 
%be the corresponding $S$-adic nilmanifold,   equipped with the unique translation-invariant probability measure %$\mu$.

Let $\X= U/\Lambda$ and $\Sol$  be the $S$-adic nilmanifold  and the associated $S$-adic solenoid
as in Section~\ref{S-Koopman Nilmanifold}.
Denote by $\mu$ the translation invariant  probability measure on 
$\X$ and let $\nu$ be the image of $\mu$ under the canonical projection 
$\vfi: \X\to \Sol.$
We identify  $L^2(\Sol)=L^2(\Sol, \nu)$ with the closed 
$\Aut(\X)$-invariant subspace 
$$\{f\circ \vfi\mid f\in L^2(\Sol)\}$$
of $L^2(\X)=L^2(\X,\mu).$
We have an orthogonal decomposition into $\Aut(\X)$-invariant 
subspaces
$$
L^2(\X)= \CCC \mathbf{1}_{\X}\oplus L_0^2(\Sol) \oplus \H, 
$$
where 
$$L_0^2(\Sol)=\{f\in L^2(\Sol)\mid \int_{\X} f d\mu=0\}$$
and where $\H$ is the orthogonal complement of $L^2(\Sol)$
in $L^2(\X).$

Let $\Ga$ be a  subgroup of  $G:=\Aut(U).$
 Let $\kappa$ be the Koopman representation of $\Ga$ on $L^2(\X)$ and 
denote by $\kappa_1$ and $\kappa_2$ the restrictions of $\kappa$
to respectively $L_0^2(\Sol)$ and  $ \H.$ 

Let $\Sigma_1$ be a set of representatives for the $\Ga$-orbits in $\widehat{\mathbf{Sol}_S}\setminus\{ \mathbf{1}_{\mathbf{Sol}_S}\}$.
We have 
$$
\kappa_1 \cong \bigoplus_{\chi\in \Sigma_1} \la_{\Ga/\Ga_\chi}, 
$$
where $\Ga_\chi$ is the stabilizer of $\chi$ in $\Ga$ and 
$\la_{\Ga/\Ga_\chi}$ is the quasi-regular representation
of $\Ga$ on $\ell^2(\Ga/\Ga_\chi).$

By Proposition~\ref{Pro-KoopmanDec}, 
there exists a family $(\pi_i)_{i\in I}$ of irreducible representations of $G,$
such that $\kappa_2$ is  equivalent to a direct sum
$$\bigoplus_{i\in I}\ind_{\Ga_i}^{\Ga}(\widetilde{\pi_i}|_{\Ga_i}\otimes W_i),$$
where  $\widetilde{\pi}_i$ is an  irreducible projective representation 
 of the stabilizer  $G_i$ of $\pi_i$ in $\Aut(G)\ltimes G$ extending $\pi_i$,
 and where  $W_i$ is a  projective unitary representation
of $\Ga_i := \Ga\cap G_i.$

\begin{proposition}
\label{Prop-ClosedSubg}
For $i\in I,$  let $\widetilde{\pi}_i$   be the (projective) representation of $G_i$  
and let $\Ga_i$  be as above.
There exists a real number  $r\geq 1$ such that  $\widetilde{\pi_i}|_{\Ga_i}$ is 
strongly $L^r$ modulo $ P_{\widetilde\pi_i}\cap \Ga_i$,
where  $P_{\widetilde\pi_i}$ is the projective kernel of $\widetilde{\pi}_i.$

 \end{proposition}

\begin{proof}
By Proposition~\ref{Pro-DecayRepAlgGr}, 
there exists a real number $r\geq 1$ such that  the representation  $\widetilde{\pi}_i$ of 
the algebraic group $G_i$  is  strongly $L^r$ modulo $P_{\widetilde\pi_i}$.
In order to show  that $\widetilde{\pi_i}|_{\Ga_i}$ is 
strongly $L^r$ modulo $ P_{\widetilde\pi_i}\cap \Ga_i$,
it suffices to show that $\Ga_i P_{\widetilde\pi_i}$ is closed in $G_i$
(compare with the proof of Proposition 6.2 in \cite{HoMo}).

Let $K_i$ be the  the  closed   $G_i$-invariant normal subgroup $K_i$  of $U$
as described in Proposition~\ref{Prop-Rational2}.
Then $\overline{\La}=K_i\La/K_i$ is a lattice in the unipotent  group $\overline{U}= U/K_i$
%and such that $P_{\pi_i}$ acts as the identity  on    $\overline{U}$. 
%Set $\overline{\mathbf{X}}=\overline{\mathbf{X}}/\overline{\La}$ and let 
By  Proposition~\ref{Prop-Rational2}.ii,  $P_{\widetilde\pi_i}$ coincides
with the kernel of the natural homomorphism
$\vfi: \Aut(U)\to \Aut(\overline{U})$.
Hence, we have 
$$\Gamma_i P_{\widetilde\pi_i}= \vfi^{-1}(\vfi (\Ga_i)).$$
Now, $\vfi (\Ga_i)$ is a discrete (and hence closed) subgroup 
of $\Aut(\overline{U})$, since $\vfi(\Ga_i)$ preserves $\overline{\La}$
(and so $\vfi( \Ga_i) \subset \Aut(\overline{U}/ \overline{\La})).$
It follows from the continuity of $\vfi$ that $\vfi^{-1}(\vfi (\Ga_i))$
is closed in $\Aut(U). \bsq$
\end{proof}

\medskip

\bigskip
\noindent
\textbf{Proof of Theorem~\ref{Theo1}}

We have to show that, if $1_\Ga$ is weakly contained in $\kappa_2,$
then $1_\Ga$ is weakly contained in $\kappa_1.$
It suffices to show that,  if $1_\Ga$ is weakly contained in $\kappa_2,$
then there exists a finite index subgroup $H$ of $\Ga$ such that 
$1_H$ is weakly contained in $\kappa_1|_{H}$ (see Theorem 2 in \cite{BekkaFrancini}).

We proceed by induction  on the integer
$$n(\Ga):=\sum_{p\in S} \dim \ZC_p(\Ga),$$
where  
$\ZC_p (\Ga)$ is the Zariski closure of the projection of $\Ga$ in $GL_n(\QQ_p)$.

If $n(\Ga)=0,$ then $\Ga$ is finite and there is nothing to prove.

Assume that $n(\Ga)\geq 1$ and that the claim above is proved for every 
 countable subgroup $H$ of $\Aut(\X)$ with  $n(H) <n(\Ga).$

 Let $I_{\rm fin} \subset I$ be the set of all $i\in I$ such that 
$\Ga_i=G_i\cap \Ga$ has finite index in $\Ga$ and set $I_{\infty}=I \setminus I_{\rm fin}.$
With $V_i=(\widetilde{\pi_i}|_{\Ga_i}\otimes W_i)$, set
$$
\kappa_2^{\rm fin} =  \bigoplus_{i\in I_{\rm fin} } \ind_{\Ga_i}^{\Ga} V_i \qquad\text{and}
\qquad
\kappa_2^\infty= \bigoplus_{i\in I_{\infty} } \ind_{\Ga_i}^{\Ga} V_i.
$$
Two cases can occur.

 \medskip
\noindent
$\bullet$ \emph{First case:}   $1_\Ga$ is weakly contained in $\kappa_2^\infty.$

Observe that  $n(\Ga_i)<n(\Ga)$ for $i\in I_{\infty}.$ Indeed,  otherwise
$\ZC_p (\Ga_i)$ and   $\ZC_p (\Ga)$ would have the same connected component $C^0_p$ 
for every $p\in S,$ since  $\Ga_i \subset  \Ga.$ Then 
$$C^0:=\bigcap_{p\in S} C^0_p $$  would  stabilize  $\pi_i$ and  $\Ga\cap C^0$ would therefore be contained in $\Ga_i.$
 Since $\Ga\cap C^0$  has finite index in $\Ga,$  this would be a contradiction to the  fact that 
 $\Ga_i$ has infinite index in $\Ga.$
 
By restriction,  $1_{\Ga_i}$ is weakly contained in $\kappa_2|_{\Ga_i}$ for every $i\in I.$
Hence, by the induction hypothesis, $1_{\Ga_i}$ is weakly contained in $\kappa_1|_{\Ga_i}$
for every $i\in I_\infty.$
Now, on the one hand, we have
$$
\kappa_1|_{\Ga_i} \cong \bigoplus_{\chi\in T_i} \la_{{\Ga_i}/\Ga_\chi\cap{\Ga_i}},
$$
for a subset $T_i$ of  $\widehat{\mathbf{Sol}_S}\setminus\{ \mathbf{1}_{\mathbf{Sol}_S}\}$.
It follows that $\Ind_{\Ga_i}^\Ga 1_{\Ga_i}=\la_{\Ga/\Ga_i}$ is weakly contained in 
%= \qquad\text{and} \qquad
$$\bigoplus_{\chi\in T_i} \Ind_{\Ga_i}^\Ga( \la_{\Ga_i/\Ga_\chi\cap \Ga_i})= 
 \bigoplus_{\chi\in T_i}
\la_{\Ga/\Ga_\chi\cap \Ga_i},
$$
 for every $i\in I_\infty.$
On the other hand, since $1_\Ga$ is weakly contained
in
$$\kappa_2\cong \bigoplus_{i\in I_\infty}\ind_{\Ga_i}^{\Ga}(\widetilde{\pi_i}|_{\Ga_i}\otimes W_i),$$
 Lemma~\ref{Lem-Herz} shows that 
 $1_\Ga$ is weakly contained in  $\bigoplus_{i\in I_\infty}\la_{\Ga/\Ga_i}.$
 It follows that $1_\Ga$  is weakly contained in 
$$ \bigoplus_{i\in I_\infty} \bigoplus_{\chi\in T_i}\la_{\Ga/\Ga_\chi\cap \Ga_i}.$$
Hence, by Lemma~\ref{Lem-Herz} again,
$1_\Ga$  is weakly contained in 
$$ \bigoplus_{i\in I_\infty} \bigoplus_{\chi\in T_i}\la_{\Ga/\Ga_\chi}.$$
This shows that $1_\Ga$  is weakly contained in $\kappa_1.$

 \medskip
\noindent
$\bullet$ \emph{Second case:}    $1_\Ga$ is weakly contained in $\kappa_2^{\rm fin}.$

By the Noetherian  property of the Zariski topology, we can find
finitely many indices $i_1, \dots, i_r$ in $I_{\rm fin}$ such that, for every $p\in S,$ we have
$$
\ZC_p(\Ga_{i_1})\cap\cdots \cap \ZC_p(\Ga_{i_r}) =\bigcap_{i\in I_{\rm fin}}  \ZC_p(\Ga_{i}),
$$
Set $H:=\Ga_{i_1}\cap \cdots\cap \Ga_{i_r}.$
Observe that $H$ has finite index in $\Ga.$
 Moreover, it follows from Lemma~\ref{Lem-RationalStab} 
 that $\ZC_p(\Ga_{i_1})\cap\cdots \cap \ZC_p(\Ga_{i_r}) $ stabilizes $\pi_{i,p}$ for every  $i\in I_{\rm fin}$ and $p\in S.$ Hence,  $H$ is contained in $\Ga_i$ for every  $i\in I_{\rm fin}$.
 
By Proposition~\ref{Pro-KoopmanDec}, we  have a decomposition of $\kappa_2^{\rm fin}|_H$
into the direct sum
$$\bigoplus_{i\in I_{\rm fin}} (\widetilde{\pi_i}\otimes W_i)|_H.$$
By Proposition \ref{Pro-LpMatrCoeff} and  Proposition~\ref{Prop-ClosedSubg},
there exists a real number $r\ge 1 ,$ which is independent of $i,$ 
such that  $(\widetilde{\pi_i}\otimes W_i)|_H$
is a strongly $L^r$ representation of $H$ modulo its projective kernel $P_i$.
Observe that $P_i$ is contained in the projective kernel $ P_{\widetilde \pi_i}$ of $ \widetilde \pi_i.$
Hence (see Proposition \ref{Pro-LpMatrCoeff}),  there exists an integer $k\geq 1$ such that 
$\kappa_2^{\rm fin}|_H^{\otimes k}$ is contained in a multiple of 
the direct sum 
$$\oplus_{i\in I_{\rm fin}} \Ind_{ P_{\widetilde \pi_i}}^H \rho_i,$$
for  representations $\rho_i$ of $ P_{\widetilde \pi_i}.$
Since $1_H$ is weakly contained in $\kappa_2^{\rm fin}|H$ and hence in $\kappa_2^{\rm fin}|_H^{\otimes k},$
using Lemma~\ref{Lem-Herz}, it follows that $1_H$ is   weakly contained  in 
$$\oplus_{i\in I_{\rm fin}} \lambda_{H/(H\cap  P_{\widetilde \pi_i})}.$$

Let $i\in I.$ We claim that  $P_i$ is contained in $\Ga_\chi$ for some character $\chi$ from
 $\widehat{\mathbf{Sol}_S}\setminus\{ \mathbf{1}_{\mathbf{Sol}_S}\}$. Once proved, this will imply,
 again by Lemma~\ref{Lem-Herz},  that $\oplus_{i\in I_{\rm fin}} \lambda_{H/(H\cap  P_{\widetilde \pi_i})}$ 
 and, hence $1_H$, is weakly contained in $\kappa_1|_H.$ 
 Since $H$ has finite index in $\Ga,$ this will show that 
 $1_\Ga$ is weakly contained in $\kappa_1$ and conclude the proof.

To prove the claim, recall from Proposition~\ref{Prop-Rational2} that 
there exists a   closed normal subgroup $K_i$ of $U$ with the following properties:
$K_i\La/K_i$ is a lattice in the unipotent algebraic group $U/K_i$, $K_i$ is invariant under $ P_{\widetilde \pi_i}$ 
and $ P_{\widetilde \pi_i}$ acts as the identity  on    $U/K_i.$ 
Observe that  $K_i\neq U$, since $\pi_i$ is not a unitary character of $U$.
We can find  a non-trivial unitary character $\chi$ of $U/K_i$  which is trivial 
on  $K_i\La/K_i$.
Then $\chi$ lifts to a non-trivial unitary character which is fixed by $ P_{\widetilde \pi_i}$
and hence by $P_i.$
$\bsq$

\section{ An example: the $S$-adic Heisenberg nilmanifold}
\label{S:Exa}
As an example, we study the spectral gap  property for  group of automorphisms  
of  the $S$-adic Heisenberg nilmanifold. We will give a quantitative estimate 
for the norm of associated convolution operators, as we did in \cite{BekkaHeu} in the
case of real Heisenberg nilmanifolds (that is, in the case  $S=\{\infty\}$).

Let $\KK$ be an algebraically closed field containing  $\QQ_p$ for 
$p=\infty$ and for all prime integers  $p$.
For an integer $n\geq 1$, consider the symplectic form $\beta$ on $\KK^{2n}$ given by
$$\beta((x,y),(x',y'))= (x,y)^t J (x',y')\tout (x,y),(x',y')\in \KK^{2n},$$
where $J$ is the $(2n\times 2n)$-matrix 
\[
J=\left(
\begin{array}{cc}
 0&I\\
-I_{n}& 0
\end{array}\
\right).
\]
The symplectic group 
$$
{Sp}_{2n}= \{g\in GL_{2n}(\KK)\mid {^{t}g}Jg=J\}
$$
is an algebraic  group defined over $\QQ.$

The $(2n+1)$-dimensional Heisenberg group is the unipotent algebraic
 group $\mathbf{H}$ defined over $\QQ,$ with underlying set $\KK^{2n}\times \KK$ and product
$$((x,y),s)((x',y'),t)=\left((x+x',y+y'),s + t +\beta((x,y),(x',y'))\right),
$$
for $(x,y), (x',y')\in \KK^{2n},\, s,t\in \KK.$

The group $Sp_{2n}$  acts by rational automorphisms of $H_{2n+1},$ given by 
$$
 g((x,y),t)= (g(x,y),t) \tout g\in Sp_{2n},\ (x,y)\in \KK^{2n},\ t\in \KK.
$$

Let $p$ be either an integer prime or $p=\infty.$
Set $H_p= \mathbf{H}(\QQ_p).$ The center $Z$ of $H_p$ is $\{(0,0, t)\mid t\in \QQ_p\}.$
The unitary dual $\widehat{H}$ of $H$ consists of the equivalence classes of the following representations:
%(see \cite[1.50]{Folland}):
\begin{itemize}
\item[$\bullet$] the unitary characters of the abelianized group $H/Z$;
\item[$\bullet$] for every $t\in\QQ_p\setminus\{0\},$ the  infinite dimensional representation 
$\pi_t$ defined on $L^2(\QQ_p^n)$ by the formula
$$
\pi_t((a,b),s)\xi(x)= e_p({ts})e_p(\langle a, x-b \rangle) \xi(x-b)
$$
for $((a,b),s)\in H,$ $\xi\in L^2(\QQ_p^n),$ and $x\in\QQ_p^n,$
where $e_p \in\widehat{\QQ_p}$ is as in Section~\ref{UniDualSolenoid}.
\end{itemize}
For $t\neq 0,$ the representation $\pi_t$
is, up to unitary equivalence, the unique 
irreducible unitary representation  of $H$ whose restriction
to the centre $Z$ is a multiple of the unitary character $s\mapsto e_p{(ts)}.$

%The group $\Aut(H)$ acts on $\widehat{H}$ by 
%$$
%\pi^g(h)= \pi(g^{-1}(h))\tout \pi\in \widehat{H}, \ g\in \Aut(H),\ h\in H.
%$$
For $g\in Sp_{2n}(\QQ_p)$ and $t\in\QQ_p\setminus\{0\},$
the representation $\pi_t^g$ is unitary equivalent to $\pi_t,$
since both representations have the same 
restriction to $Z.$ This shows that $Sp_{2n}(\QQ_p)$ stabilizes
$\pi_t$. We denote the corresponding projective representation 
of $Sp_{2n}(\QQ_p)$ by $\omega_t^{(p)}$.
The representation $\omega_t^{(p)}$ has different names: it is called \textbf{metaplectic representation, Weil's representation, }  or \textbf{oscillator representation.}
The projective kernel of  $\omega_t^{(p)}$ 
coincides with the (finite) center of  $Sp_{2n}(\QQ_p)$  and 
$\omega_t^{(p)}$ is strongly $L^{4n+2+ \eps}$ 
on $Sp_{2n}(\QQ_p)$ for every $\eps>0$ (see  Proposition 6.4 in \cite{HoMo} or Proposition 8.1 in \cite{Howe3}).

Let $S= \{p_1, \dots, p_r, \infty\}$, where $p_1, \dots, p_r$ are integer primes. 
Set  $U:=\mathbf{H}(\QQ_S)$   and  
$$\La:=\mathbf{H}(\ZZ[1/S])=\{((x,y),s)\ :\ x,y\in \ZZ^n[1/S], s\in \ZZ[1/S]\}.$$
%Let   $\X:= \mathbf{U}(\QQ_S)/\mathbf{U}(\ZZ[1/S])$ 
%be the corresponding $S$-adic nilmanifold,   equipped with the unique translation-invariant probability measure %$\mu$.
Let $\X= U/\Lambda$;
the associated $S$-adic solenoid is $\Sol= \QQ_S^{2n}/\ZZ[1/S]^{2n}.$ 
The group $Sp_{2n}(\ZZ[1/S])$ is a subgroup of $\Aut(\X)$.
The action of $Sp_{2n}(\ZZ[1/S])$ on $\Sol$ is induced by its  linear representation
by linear bijections on $\QQ_S$.

Let $\Ga$ be a subgroup of $Sp_{2n}(\ZZ[1/S])$.
 The  Koopman representation $\kappa$ 
of $\Ga$ on $L^2(\X)$ decomposes as  
$$\kappa=\mathbf{1}_{\X}\oplus \kappa_1\oplus \kappa_2,
$$
where $\kappa_1$ is the restriction of $\kappa$
to $L_0^2(\Sol)$
and $\kappa_1$  the restriction of $\kappa$ to the orthogonal complement of $L^2(\Sol)$
in $L^2(\X).$
Since  $Sp_{2n}(\QQ_p)$ stabilizes every  infinite dimensional representation of 
$H_p,$ it follows from Proposition~\ref{Prop-Rational2} that there exists a subset 
$I\subset \QQ$ such that 
 $\kappa_2$ is  equivalent to a direct sum
$$\bigoplus_{t\in I}\left(\otimes_{p\in S}(\omega_t^{(p)}|_\Ga\otimes W_i)\right),$$
where  $ W_i$ is an   projective representation $\Ga.$

Let $\nu$ be a  probability  measure  on $\Ga.$ We can give an estimate
of the norm of $\kappa_2(\nu)$ as in \cite{BekkaHeu} in the case of 
$S=\{\infty\}.$ Indeed, by a general inequality (see Proposition 30 in \cite{BachirYves}),
we have 
$$
\Vert \kappa_2(\nu)\Vert \leq \Vert \left(\kappa_2\otimes\overline{\kappa_2}\right)^{\otimes k}(\nu)\Vert^{1/2k},
$$
for every integer $k\geq 1,$ where $\overline{\kappa_2}$ denotes the representation conjugate to  $\kappa_2$.
Since  $\omega_t^{(p)}$ is strongly $L^{4n+2+ \eps}$ 
on $Sp_{2n}(\QQ_p)$ for any   $t\in I$ and $p\in S,$  
 Proposition~\ref{Pro-LpMatrCoeff} implies  that $(\kappa_2\otimes\overline{\kappa_2})^{\otimes (n+1)}$
  is contained in
 an infinite multiple of the regular representation $\la_\Ga$ of $\Ga.$
 Hence, 
$$
\Vert \kappa_2(\nu)\Vert \leq \Vert \la_\Ga(\nu)\Vert^{1/2n+2}
$$
and so, 
$$
\Vert \kappa_0(\nu)\Vert \leq \max\{\Vert \la_\Ga(\nu)\Vert^{1/2n+2}, \Vert \kappa_1(\nu)\Vert\},
$$
where $\kappa_0$ is the restriction of $\kappa$ to $L^2_0(\X).$

Assume that the  subgroup generated by the support of $\nu$  coincides with $\Ga.$
If $\Ga$ is not amenable then  $\Vert \la_\Ga(\nu)\Vert <1$ by  Kersten's theorem (see 
\cite[Appendix G]{BHV});
so, in this case, the action of $\Ga$ on $\X$ has a spectral gap if and only if 
$\Vert \kappa_1(\nu)\Vert<1,$ as stated in Theorem~\ref{Theo1}.

Observe that, if $\Ga$ is amenable, then  the action of $\Ga$ on $\X$ or $\Sol$ does not have a spectral gap;
indeed, by   a general result  (see \cite[Theorem 2.4]{JuRo}),
 no  action of  a countable amenable group  by measure preserving transformations on a 
 non-atomic probability space  has a spectral gap.

Let us look more closely to the case $n=1.$ We have  
$Sp_{2}(\ZZ[1/S])=SL_2(\ZZ[1/S]$
and the stabilizer of every element in $\widehat{\mathbf{Sol}_S}\setminus\{ \mathbf{1}_{\mathbf{Sol}_S}\}$
is conjugate to the group of unipotent matrices  in $SL_2(\ZZ[1/S]$ and hence amenable. 
This implies that $\kappa_1$ is weakly contained in $\la_\Ga;$
so,  we have 
$$\Vert \kappa_1(\nu)\Vert<1 \Longleftrightarrow \Ga \text{ is not amenable.} $$
As a consequence, we see that the action of $\Ga$ on $\X$ has a spectral gap if and only if 
$\Ga$ is not amenable.

\end{document}